\documentclass[12pt, reqno]{amsart}
\usepackage{amsmath, amsthm, amscd, amsfonts, amssymb, graphicx, color}
\usepackage[bookmarksnumbered, colorlinks, plainpages]{hyperref}

\input{mathrsfs.sty}

\newtheorem{theorem}{Theorem}[section]
\newtheorem{lemma}[theorem]{Lemma}
\newtheorem{proposition}[theorem]{Proposition}
\newtheorem{corollary}[theorem]{Corollary}
\theoremstyle{definition}

\theoremstyle{remark}
\newtheorem{remark}[theorem]{Remark}
\numberwithin{equation}{section}
\begin{document}

\title [Some reversed and refined Callebaut inequalities ]{Some reversed and refined Callebaut inequalities via Kontorovich constant}

\author[M. Bakherad]{Mojtaba Bakherad}

\address{Department of Mathematics, Faculty of Mathematics, University of Sistan and Baluchestan, Zahedan, Iran.}
\email{mojtaba.bakherad@yahoo.com; bakherad@member.ams.org}



%~~~~~~~~~~~~~~~~~~~~~~~~~~~~~~~~~~~~~~~~~~~~~~~~~~~~~~~~~~~~~~~~~~~~~~~~~~~~~~~~~~~~~~~~~~~~~~~~~~~~~~~~~~~~~~~~~~~~~~~~~~~~~~~~~~~
\begin{abstract}
{In this paper we employ some operator techniques to establish some refinements and reverses of the Callebaut  inequality  involving the geometric mean and Hadamard product under some mild conditions. In particular, we show
 \begin{align*}
K&\left(\frac{M^{2t-1}}{m^{2t-1}},2\right)^{r'}
\sum_{j=1}^n(A_j\sharp_{s}B_j)\circ \sum_{j=1}^n(A_j\sharp_{1-s}B_j)
\nonumber\\&\,\,+\left(\frac{t-s}{t-1/2}\right)\left(\sum_{j=1}^n(A_j\sharp_{t}B_j)\circ \sum_{j=1}^n(A_j\sharp_{1-t}B_j)
-\sum_{j=1}^n(A_j\sharp B_j)\circ \sum_{j=1}^n(A_j\sharp B_j)\right)\nonumber
\\&\leq \sum_{j=1}^n(A_j\sharp_{t}B_j)\circ \sum_{j=1}^n(A_j\sharp_{1-t} B_j)\,,
\end{align*}
where $A_j, B_j\in{\mathbb B}({\mathscr H})\,\,(1\leq j\leq n)$ are positive operators such that $0<m' \leq B_j\leq m <M \leq A_j\leq M'\,\,(1\leq j\leq n)$, either $1\geq t\geq s>{\frac{1}{2}}$ or $0\leq t\leq s<\frac{1}{2}$, $r'=\min\left\{\frac{t-s}{t-1/2},\frac{s-1/2}{t-1/2}\right\}$ and $K(t,2)=\frac{(t+1)^2}{4t}\,\,(t>0)$.}
\end{abstract} \maketitle

%~~~~~~~~~~~~~~~~~~~~~~~~~~~~~~~~~~~~~~~~~~~~~~~~~~~~~~~~~~~~~~~~~~~~~~~~~~~~~~~~~~~~~~~~~~~~~~~~~~~~~~~~~~~~~~~~~~~~~~~~~~~~~~~~~~~
\section{Introduction and preliminaries}
Let ${\mathbb B}({\mathscr H})$ denote the $C^*$-algebra of all
bounded linear operators on a complex Hilbert space ${\mathscr H}$ with the identity $I$.
An operator $A\in{\mathbb B}({\mathscr H})$ is called positive
if $\langle Ax,x\rangle\geq0$ for all $x\in{\mathscr H }$, and we then write $A\geq0$. We write $A>0$ if $A$ is
a positive invertible operator. The set of all positive invertible operators  is denoted by ${\mathbb B}({\mathscr H})_+$. For
self-adjoint operators $A, B\in{\mathbb B}({\mathscr H})$, we say
$B\geq A$ if $B-A\geq0$. The Gelfand map $f\mapsto f(A)$ is an
isometric $*$-isomorphism between the $C^*$-algebra
$C(\rm{sp}(A))$ of a complex-valued continuous functions on the spectrum $\rm{sp}(A)$
of a self-adjoint operator $A$ and the $C^*$-algebra generated by $I$ and $A$. If $f, g\in C({\rm sp}(A))$, then
$f(t)\geq g(t)\,\,(t\in{\rm sp}(A))$ implies that $f(A)\geq g(A)$.

It is known that the Hadamard
product can be presented by filtering the
tensor product $A \otimes B$ through a positive linear map. In fact,
 $A\circ B=U^*(A\otimes B)U$,
where $U:{\mathscr H}\to {\mathscr H}\otimes{\mathscr H}$ is the isometry
defined by $Ue_j=e_j\otimes e_j$, where $(e_j)$ is an orthonormal basis of the Hilbert space ${\mathscr H}$; see \cite{paul}.

For $A, B\in{\mathbb B}({\mathscr H})_+$, the operator geometric mean $A\sharp B$ is defined by
$A\sharp B=A^{\frac{1}{2}}\left(A^{\frac{-1}{2}}BA^{\frac{-1}{2}}\right)^{\frac{1}{2}}A^{\frac{1}{2}}.$
    For $\alpha\in(0,1)$,  the operator weighted geometric mean is defined by
\begin{align*}
A\sharp_\alpha
B=A^{\frac{1}{2}}\left(A^{\frac{-1}{2}}BA^{\frac{-1}{2}}\right)^{\alpha}A^{\frac{1}{2}}.
\end{align*}
Callebaut \cite{CAL} showed the following refinement of the Cauchy--Schwarz inequality
\begin{align}\label{choshi}
\left(\sum_{j=1}^n x_j^{\frac{1}{2}}y_j^{\frac{1}{2}}\right)^2&\leq
\sum_{j=1}^n x_j^{\frac{1+s}{2}}y_j^{\frac{1-s}{2}}\sum_{j=1}^n x_j^{\frac{1-s}{2}}y_j^{\frac{1+s}{2}}\nonumber
\\&\leq\sum_{j=1}^n x_j^{\frac{1+t}{2}}y_j^{\frac{1-t}{2}}\sum_{j=1}^n x_j^{\frac{1-t}{2}}y_j^{\frac{1+t}{2}}\nonumber
\\&\leq \left(\sum_{j=1}^n x_j\right)\left(\sum_{j=1}^ny_j\right),
 \end{align}
where $ x_j, y_j \,\,(1\leq j\leq n)$ are positive real numbers and either $1\geq t\geq s>{\frac{1}{2}}$ or $0\leq t\leq s<\frac{1}{2}$. This is indeed an extension of the Cauchy--Schwarz inequality.\\
 Wada \cite{wada} gave an
operator version of the Callebaut inequality by showing that if $A, B\in{\mathbb B}({\mathscr H})_+$, then
{\begin{align*}
(A\sharp B)\otimes(A\sharp B)&\leq\frac{1}{2}\left\{(A\sharp_\alpha B)\otimes (A\sharp_{1-\alpha} B)+(A\sharp_{1-\alpha} B)\otimes(A\sharp_{\alpha} B)\right\}\\&\leq \frac{1}{2}\left\{(A\otimes B)+(B\otimes A)\right\},
\end{align*} }
where $\alpha\in[0,1]$. In \cite{caleba} the authors showed another operator version of the Callebaut inequality
\begin{align}\label{34rf}
\sum_{j=1}^n(A_j\sharp B_j)\circ \sum_{j=1}^n(A_j\sharp B_j)\nonumber
&\leq\sum_{j=1}^n(A_j\sharp_{s}B_j)\circ \sum_{j=1}^n(A_j\sharp_{1-s}B_j)\nonumber
\\&\leq \sum_{j=1}^n(A_j\sharp_tB_j)\circ \sum_{j=1}^n(A_j\sharp_{1-t} B_j)\nonumber\\&\leq\left(\sum_{j=1}^nA_j\right)\circ \left(\sum_{j=1}^nB_j\right)\,,
\end{align}
where $A_j, B_j\in{\mathbb B}({\mathscr H})_+\,\,(1\leq j\leq n)$ and either $1\geq t\geq s>{\frac{1}{2}}$ or $0\leq t\leq s<\frac{1}{2}$.\\
In \cite{mo-moj} the authors presented the following refinement of inequality \eqref{34rf} as follows
{\footnotesize\begin{align}\label{moj-mo}
\sum_{j=1}^n(A_j\sharp_{s}B_j)&\circ \sum_{j=1}^n(A_j\sharp_{1-s}B_j)\nonumber
\\&\leq\sum_{j=1}^n(A_j\sharp_{s}B_j)\circ \sum_{j=1}^n(A_j\sharp_{1-s}B_j)\nonumber
\\&\,\,+\left(\frac{t-s}{s-1/2}\right)\left(\sum_{j=1}^n(A_j\sharp_{s}B_j)\circ \sum_{j=1}^n(A_j\sharp_{1-s}B_j)
-\sum_{j=1}^n(A_j\sharp B_j)\circ \sum_{j=1}^n(A_j\sharp B_j)\right)\nonumber
\\&\leq \sum_{j=1}^n(A_j\sharp_{t}B_j)\circ \sum_{j=1}^n(A_j\sharp_{1-t} B_j)\,,
 \end{align}}
in which $A_j, B_j\in{\mathbb B}({\mathscr H})_+\,\,(1\leq j\leq n)$ and either $1\geq t\geq s>{\frac{1}{2}}$ or $0\leq t\leq s<\frac{1}{2}$.\\
There have been obtained several Cauchy--Schwarz type inequalities for Hilbert space operators and matrices; see
 \cite{aldaz, ABM, I-V, salemi} and references therein.

In this paper, we  present some  refinements and reverses of the Callebaut inequality involving the weighted geometric mean and Hadamard product of Hilbert space operators.
%===================================================================================================================================
\section{Further refienements of the Callebaut inequality involving Hadamard product }

The Kontorovich constant is
\begin{align*}
K(t,2)=\frac{(t+1)^2}{4t}\qquad(t>0).
\end{align*}
The classical Young inequality states that
 \begin{align*}
 a^\nu b^{1-\nu}\leq \nu a+(1-\nu)b,
\end{align*}
where $a,b\geq0$ and $\nu\in[0,1]$.
Recently, Zuo et. al. \cite{zho} showed an improvement of the Young inequality as follows:
\begin{align*}
K\left(\sqrt{\frac{a}{b}},2\right)^{r}a^\nu b^{1-\nu}\leq \nu a+(1-\nu)b,
\end{align*}
where $a,b>0$, $\nu\in[0,1]$, $r=\min\left\{\nu,1-\nu\right\}$.
Applying this inequality, J. Wu and J. Zhao \cite{W-Z} showed the following refiniment of the Young inequality
\begin{align}\label{Wu-Zhao}
K\left(\sqrt{\frac{a}{b}},2\right)^{r'}a^\nu b^{1-\nu}+r\left(
\sqrt{a}-\sqrt{b}\right)^2\leq \nu a+(1-\nu)b,
\end{align}
where $a,b>0$, $\nu\in[0,1]-\left\{\frac{1}{2}\right\}$, $r=\min\left\{\nu,1-\nu\right\}$ and $r'=\min\left\{2r,1-2r\right\}$.
Using \eqref{Wu-Zhao} we get  the following lemmas.
\begin{lemma}
Let $a,b>0$ and $\nu\in[0,1]-\left\{\frac{1}{2}\right\}$. Then
\begin{align}\label{kokol}
K\left(\sqrt{\frac{a}{b}},2\right)^{r'}\left(a^\nu b^{1-\nu}+a^{1-\nu}b^{\nu}\right)+2r\left(
\sqrt{a}-\sqrt{b}\right)^2\leq a+b,
\end{align}
where  $r=\min\left\{\nu,1-\nu\right\}$ and $r'=\min\left\{2r,1-2r\right\}.$
\end{lemma}
\begin{lemma}\label{rar32}
Let $0<m' \leq B\leq m <M \leq A\leq M'$ and either $1\geq t\geq s>{\frac{1}{2}}$ or $0\leq t\leq s<\frac{1}{2}$.  Then
\begin{align}\label{tool}
K\left(\frac{M^{2t-1}}{m^{2t-1}},2\right)^{r'}
&(A^{s}\otimes B^{1-s}+A^{1-s}\otimes B^{s}) \nonumber\\&\,\,\,\,+\left(\frac{t-s}{t-1/2}\right)\left(A^{t}\otimes B^{1-t}+A^{1-t}\otimes B^{t}-2(A^{\frac{1}{2}}\otimes B^{\frac{1}{2}})\right) \nonumber\\&\leq A^{t}\otimes B^{1-t}+A^{1-t}\otimes B^{t}\,,
 \end{align}
where $r'=\min\left\{\frac{t-s}{t-1/2},\frac{s-1/2}{t-1/2}\right\}$.
\end{lemma}
\begin{proof}
Let $a>0$. If we replace $b$ by $a^{-1}$ and take $\nu=\frac{1-\mu}{2}$  \eqref{kokol}, then we get
\begin{align}\label{ttt1}
K(a,2)^{r'}\left(a^\mu+a^{-\mu}\right)+(1-\mu)\left(
a+a^{-1}-2\right)\leq a+a^{-1},
\end{align}
in which  $\mu\in(0,1]$ and  $r'=\min\left\{1-\mu,\mu\right\}$. Let us fix positive real numbers $\alpha,\beta$ such that $\beta<\alpha$. It follows from $0<m' \leq B\leq m <M \leq A\leq M'$ that
$I\leq h=\left (\frac{M}{m}\right)^\alpha\leq A^\alpha\otimes B^{-\alpha}\leq h'=\left(\frac{M'}{m'}\right)^\alpha$ and $\rm{sp}(A^\alpha\otimes B^{-\alpha})\subseteq[h,h']\subseteq(1,+\infty)$.
Since the Kontorovich constant $\frac{(a+1)^2}{4a}$ is an increasing function on $(1,+\infty)$, by \eqref{ttt1} we have
\begin{align*}
K\left(\frac{M^\alpha}{m^\alpha},2\right)^{r'}\left(a^\mu+a^{-\mu}\right)+(1-\mu)\left(
a+a^{-1}-2\right)\leq a+a^{-1},
\end{align*}
 where $\mu\in(0,1]$ and  $r'=\min\left\{1-\mu,\mu\right\}$. Using the functional calculus,
if we replace $a$ by  the operator $A^\alpha\otimes B^{-\alpha}$ and $\mu$ by $\frac{\beta}{\alpha}$ we have
\begin{align}\label{theo345}
K\left(\frac{M^\alpha}{m^\alpha},2\right)^{r'}&\left(A^\beta\otimes B^{-\beta}+A^{-\beta}\otimes B^{\beta}\right)\nonumber\\&\,\,\,+\left(1-\frac{\beta}{\alpha}\right)\left(
A^\alpha\otimes B^{-\alpha}+A^{-\alpha}\otimes B^{\alpha}-2I\right)\nonumber\\&\leq A^\alpha\otimes B^{-\alpha}+A^{-\alpha}\otimes B^{\alpha},
\end{align}
where  $r'=\min\left\{1-\frac{\beta}{\alpha},\frac{\beta}{\alpha}\right\}$. Multiplying both sides of \eqref{theo345} by $A^\frac{1}{2}\otimes B^\frac{1}{2}$ we reach
\begin{align}\label{theo3456}
K\left(\frac{M^{\alpha}}{m^{\alpha}},2\right)^{r'}&\left(A^{1+\beta}\otimes B^{1-\beta}+A^{1-\beta}\otimes B^{1+\beta}\right)\nonumber\\&\,\,\,+\left(1-\frac{\beta}{\alpha}\right)\left(
A^{1+\alpha}\otimes B^{1-\alpha}+A^{1-\alpha}\otimes B^{1+\alpha}-2(A\otimes B)\right)\nonumber\\&\leq A^{1+\alpha}\otimes B^{1-\alpha}+A^{1-\alpha}\otimes B^{1+\alpha}.
\end{align}
Now, if we replace $\alpha, \beta, A, B$ by $2t-1, 2s-1, A^\frac{1}{2}, B^\frac{1}{2}$ respectively, in \eqref{theo3456}, we
obtain
\begin{align*}
K\left(\frac{M^{2t-1}}{m^{2t-1}},2\right)^{r'
}&(A^{s}\otimes B^{1-s}+A^{1-s}\otimes B^{s}) \\&\,\,\,\,+\left(\frac{t-s}{t-1/2}\right)\left(A^{t}\otimes B^{1-t}+A^{1-t}\otimes B^{t}-2(A^{\frac{1}{2}}\otimes B^{\frac{1}{2}})\right) \\&\leq A^{t}\otimes B^{1-t}+A^{1-t}\otimes B^{t}\,
\end{align*}
for  either $1\geq t\geq s>{\frac{1}{2}}$ or $0\leq t\leq s<\frac{1}{2}$ and $r'=\min\left\{\frac{t-s}{t-1/2},\frac{s-1/2}{t-1/2}\right\}$.
\end{proof}
We are ready to prove the first result of this section.
\begin{theorem}\label{vow13}
Let $0<m' \leq B_j\leq m <M \leq A_j\leq M'\,\,(1\leq j\leq n)$ and either $1\geq t\geq s>{\frac{1}{2}}$ or $0\leq t\leq s<\frac{1}{2}$.  Then
\begin{align}\label{maman}
K&\left(\frac{M^{2t-1}}{m^{2t-1}},2\right)^{r'}
\sum_{j=1}^n(A_j\sharp_{s}B_j)\circ \sum_{j=1}^n(A_j\sharp_{1-s}B_j)
\nonumber\\&\,\,+\left(\frac{t-s}{t-1/2}\right)\left(\sum_{j=1}^n(A_j\sharp_{t}B_j)\circ \sum_{j=1}^n(A_j\sharp_{1-t}B_j)
-\sum_{j=1}^n(A_j\sharp B_j)\circ \sum_{j=1}^n(A_j\sharp B_j)\right)\nonumber
\\&\leq \sum_{j=1}^n(A_j\sharp_{t}B_j)\circ \sum_{j=1}^n(A_j\sharp_{1-t} B_j)\,,
\end{align}
where $r'=\min\left\{\frac{t-s}{t-1/2},\frac{s-1/2}{t-1/2}\right\}$.
\end{theorem}
\begin{proof}
Put $C_j=A_j^{-{1\over2}}B_jA_j^{-{1\over2}}\,\,(1\leq j\leq n)$. By inequality \eqref{tool} we get
\begin{align}\label{evs}
K\left(\frac{M^{2t-1}}{m^{2t-1}},2\right)^{r'}&\left(C_j^{s}\otimes C_i^{1-s}+C_j^{1-s}\otimes C_i^{s}\right)\nonumber\\&\,\,\,\,+\left(\frac{t-s}{t-1/2}\right)
\left(C_j^{t}\otimes C_i^{1-t}+C_j^{1-t}\otimes C_i^{t}-2\big(C_j^{\frac{1}{2}}\otimes C_i^{\frac{1}{2}}\big)\right)\nonumber\\&\leq C_j^{t}\otimes C_i^{1-t}+C_j^{1-t}\otimes C_i^{t}\qquad(1\leq i,j\leq n).
 \end{align}
Multiplying both sides of \eqref{evs} by $A_j^{\frac{1}{2}}\otimes A_i^{\frac{1}{2}}$ we get

 \begin{align}\label{9090}
K\left(\frac{M^{2t-1}}{m^{2t-1}},2\right)^{r'}&\left((A_j\sharp_{s} B_j)\otimes (A_i\sharp_{1-s} B_i)+(A_j\sharp_{1-s} B_j)\otimes (A_i\sharp_{s}B_i)\right)\nonumber\\&
\,\,\,\,+\left(\frac{t-s}{t-1/2}\right)
\Big((A_j\sharp_{t} B_j)\otimes (A_i\sharp_{1-t} B_i)+(A_j\sharp_{1-t} B_j)\nonumber\\&\,\,\,\,\otimes (A_i\sharp_{t} B_i)-2(A_j\sharp B_j)\otimes (A_i\sharp B_i)\Big)\nonumber\\&\leq
(A_j\sharp_{t}B_j)\otimes (A_i\sharp_{1-t} B_i)+(A_j\sharp_{1-t}B_j)\otimes (A_i\sharp_{t}B_i)\,
 \end{align}
for all $1\leq i,j\leq n$. Therefore
 {\footnotesize\begin{align*}
&K\left(\frac{M^{2t-1}}{m^{2t-1}},2\right)^{r'}\left(\sum_{j=1}^n(A_j\sharp_{s}B_j)\circ \sum_{j=1}^n(A_j\sharp_{1-s}B_j)\right)
\\ & \,+\left(\frac{t-s}{t-1/2}\right)\left(\sum_{j=1}^n(A_j\sharp_{t}B_j)\circ \sum_{j=1}^n(A_j\sharp_{1-t}B_j)
-\Big(\sum_{j=1}^nA_j\sharp B_j\Big)\circ\Big(\sum_{j=1}^n A_j\sharp B_j\Big)\right)\\&=
\frac{1}{2}\sum_{i,j=1}^n\Big[K\left(\frac{M^{2t-1}}{m^{2t-1}},2\right)^{r'}\Big((A_j\sharp_{s} B_j)\circ (A_i\sharp_{1-s} B_i)+(A_j\sharp_{1-s} B_j)\circ (A_i\sharp_{s}B_i)\Big)\nonumber\\&
\,+\left(\frac{t-s}{t-1/2}\right)
\left((A_j\sharp_{t} B_j)\circ (A_i\sharp_{1-t} B_i)+(A_j\sharp_{1-t} B_j)\circ (A_i\sharp_{t} B_i)-2(A_j\sharp B_j)\circ (A_i\sharp B_i)\right)\Big]\\&\leq \frac{1}{2}\sum_{i,j=1}^n\left((A_j\sharp_{t}B_j)\circ (A_i\sharp_{1-t} B_i)+(A_j\sharp_{1-t}B_j)\circ (A_i\sharp_{t}B_i)\right)\qquad(\textrm{by inequality\,\eqref{9090}})\\&
=\sum_{j=1}^n(A_j\sharp_{t}B_j)\circ \sum_{j=1}^n(A_j\sharp_{1-t} B_j)\,.
\end{align*}}
\end{proof}
\begin{remark}
It follows from $$\left(\frac{t-s}{t-1/2}\right)\left(\sum_{j=1}^n(A_j\sharp_{t}B_j)\circ \sum_{j=1}^n(A_j\sharp_{1-t}B_j)
-\sum_{j=1}^n(A_j\sharp B_j)\circ \sum_{j=1}^n(A_j\sharp B_j)\right)\geq0\,,$$ where  $A_j, B_j\in {\mathbb B}({\mathscr H})_+\,\,(1\leq j\leq n)$, either $1\geq t\geq s>{\frac{1}{2}}$ or $0\leq t\leq s<\frac{1}{2}$ and $K(t,2)=\frac{(t+1)^2}{4t}\geq 1\,\,(t>0)$ that inequality \eqref{maman} is a refinement of the second inequality of inequalities \eqref{34rf} and \eqref{moj-mo}.
\end{remark}

{We conclude an application of Theorem \ref{vow13} for numerical cases which is a refinement of inequality \eqref{choshi}.
\begin{corollary}
Let $0<m' \leq y_j\leq m <M \leq x_j\leq M'\,\,(1\leq j\leq n)$ and either  $-1\leq t\leq s<0$ or $1\geq t\geq s>0$.  Then
\begin{align*}
&\sum_{j=1}^nx_j^{{1+s}\over2}y_j^{{1-s\over2}} \sum_{j=1}^nx_j^{{1-s}\over2}y_j^{{1+s\over2}}
\\&\leq K\left(\frac{M^{t}}{m^{t}},2\right)^{r'}
\sum_{j=1}^nx_j^{{1+s}\over2}y_j^{{1-s\over2}} \sum_{j=1}^nx_j^{{1-s}\over2}y_j^{{1+s\over2}}
\\&\,\,+\left(\frac{t-s}{t}\right)\left(\sum_{j=1}^nx_j^{{1+t}\over2}y_j^{{1-t\over2}} \sum_{j=1}^nx_j^{{1-t}\over2}y_j^{{1+t\over2}}
-\left(\sum_{j=1}^nx_j^{1\over2}y_j^{1\over2}\right)^2 \right)
\\&\leq \sum_{j=1}^nx_j^{{1+t}\over2}y_j^{{1-t\over2}} \sum_{j=1}^nx_j^{{1-t}\over2}y_j^{{1+t\over2}}\,,
\end{align*}
where $r'=\min\left\{\frac{t-s}{t},\frac{s}{t}\right\}$.
\end{corollary}}
\begin{proof}
{If we put $A_j=x_j, B_j=y_j\,\,(1\leq j\leq n)$ in Theorem \ref{vow13} and inequality \eqref{34rf}, then we get
\begin{align*}
&\sum_{j=1}^nx_j^{{1-s}}y_j^{{s}} \sum_{j=1}^nx_j^{{s}}y_j^{{1-s}}
\\&\leq K\left(\frac{M^{2t-1}}{m^{2t-1}},2\right)^{r'}
\sum_{j=1}^nx_j^{{1-s}}y_j^{{s}} \sum_{j=1}^nx_j^{{s}}y_j^{{1-s}}
\\&\,\,+\left(\frac{t-s}{t-1/2}\right)\left(\sum_{j=1}^nx_j^{{1-t}}y_j^{{t}} \sum_{j=1}^nx_j^{{t}}y_j^{{1-t}}
-\left(\sum_{j=1}^nx_j^{1\over2}y_j^{1\over2}\right)^2 \right)
\\&\leq \sum_{j=1}^nx_j^{{1-t}}y_j^{{t}} \sum_{j=1}^nx_j^{{t}}y_j^{{1-t}}\,,
\end{align*}
where $r'=\min\left\{\frac{t-s}{t-1/2},\frac{s-1/2}{t-1/2}\right\}$ and either  $0\leq t\leq s<\frac{1}{2}$ or $1\geq t\geq s>{\frac{1}{2}}$. Now if we replace $s$ by ${s+1\over2}$ and $t$ by ${t+1\over2}$, respectively, where either  $-1\leq t\leq s<0$ or $1\geq t\geq s>0$, then we reach the desired inequalities.}
\end{proof}

\begin{lemma}
Let $a,b>0$ and $\nu\in(0,1)$. Then
\begin{align*}
a^\nu b^{1-\nu}+a^{1-\nu} b^\nu +2r(\sqrt{a}-\sqrt{b})^2+r'\left(2\sqrt{ab}+a+b-2a^{\frac{1}{4}}b^{\frac{3}{4}}-2a^{\frac{3}{4}}b^{\frac{1}{4}}\right) \leq a+b\,,
\end{align*}
where $r=\min\{\nu,1-\nu\}$ and $r'=\min\{2r,1-2r\}$.
\end{lemma}
\begin{proof}
Let $a,b>0$, $\nu\in(0,1)$, $r=\min\{\nu,1-\nu\}$ and $r'=\min\{2r,1-2r\}$. Applying  \cite[Lemma 1]{zha}, we have the inequalities
 $$a^{1-\nu} b^\nu+\nu(\sqrt{a}-\sqrt{b})^2+r'(\sqrt[4]{ab}-\sqrt{a})^2 \leq(1-\nu) a+\nu b\,,$$
where $0<\nu\leq\frac{1}{2}$ and
$$a^{1-\nu} b^\nu+(1-\nu)(\sqrt{a}-\sqrt{b})^2+r'(\sqrt[4]{ab}-\sqrt{b})^2 \leq(1-\nu) a+\nu b\,,$$
where  $\frac{1}{2}<\nu<1$. Summing these inequalities we get the desired result.
\end{proof}
\begin{lemma}\label{rar3234}
Let $A_j, B_j\in{\mathbb B}({\mathscr H})_+\,\,(1\leq j\leq n)$ and either $1\geq t\geq s>{\frac{1}{2}}$ or $0\leq t\leq s<\frac{1}{2}$.  Then
\begin{align*}
&(A^{s}\otimes B^{1-s}+A^{1-s}\otimes B^{s}) \\&\,\,\,\,+\left(\frac{t-s}{t-1/2}\right)\left(A^{s}\otimes B^{1-s}+A^{1-s}\otimes B^{s}-2(A^{\frac{1}{2}}\otimes B^{\frac{1}{2}})\right)\\&\,\,\,\,+r'\left(A^{t}\otimes B^{1-s}+A^{1-s}\otimes B^{s}+2(A^{\frac{1}{2}}\otimes B^{\frac{1}{2}})-2A^{\frac{1+2s}{4}}\otimes B^{\frac{3-2s}{4}}-2A^{\frac{3-2s}{4}}\otimes B^{\frac{1+2s}{4}}\right) \\&\leq A^{t}\otimes B^{1-t}+A^{1-t}\otimes B^{t}\,,
 \end{align*}
where $r'=\min\left\{\frac{t-s}{t-1/2},\frac{s-1/2}{t-1/2}\right\}$.
\end{lemma}
Using Lemma \ref{rar3234} and the same argument  in the proof of Lemma \ref{rar32}, we get another refinement of inequality  \eqref{34rf}.

\begin{theorem}\label{okmn345}
Let $A_j, B_j\in{\mathbb B}({\mathscr H})_+\,\,(1\leq j\leq n)$ and either $1\geq t\geq s>{\frac{1}{2}}$ or $0\leq t\leq s<\frac{1}{2}$. Then
\begin{align}\label{maman2}
&\sum_{j=1}^n(A_j\sharp_{s}B_j)\circ \sum_{j=1}^n(A_j\sharp_{1-s}B_j)\nonumber
\\&\,\,+\left(\frac{t-s}{t-1/2}\right)\left(\sum_{j=1}^n(A_j\sharp_{s}B_j)\circ \sum_{j=1}^n(A_j\sharp_{1-s}B_j)
-\sum_{j=1}^n(A_j\sharp B_j)\circ \sum_{j=1}^n(A_j\sharp B_j)\right)\nonumber
\\&\,\,+r'\Big(\sum_{j=1}^n(A_j\sharp_{s}B_j)\circ \sum_{j=1}^n(A_j\sharp_{1-s}B_j)
+\sum_{j=1}^n(A_j\sharp B_j)\circ \sum_{j=1}^n(A_j\sharp B_j)\nonumber\\&\,\,\,\,-2\sum_{j=1}^n(A_j\sharp_{\frac{3-2s}{4}}B_j)\circ \sum_{j=1}^n(A_j\sharp_{\frac{1+2s}{4}}B_j)\Big)\nonumber
\\&\leq \sum_{j=1}^n(A_j\sharp_{t}B_j)\circ \sum_{j=1}^n(A_j\sharp_{1-t} B_j)\,,
 \end{align}
where $r'=\min\left\{\frac{t-s}{t-1/2},\frac{s-1/2}{t-1/2}\right\}$.
\end{theorem}

\begin{remark}
If $A_j, B_j\in{\mathbb B}({\mathscr H})_+\,\,(1\leq j\leq n)$ and either $1\geq t\geq s>{\frac{1}{2}}$ or $0\leq t\leq s<\frac{1}{2}$, then
{\footnotesize\begin{align*}\sum_{j=1}^n(A_j\sharp_{s}B_j)\circ \sum_{j=1}^n(A_j\sharp_{1-s}B_j)
+\sum_{j=1}^n(A_j\sharp B_j)\circ \sum_{j=1}^n(A_j\sharp B_j)-2\sum_{j=1}^n(A_j\sharp_{\frac{3-2s}{4}}B_j)\circ \sum_{j=1}^n(A_j\sharp_{\frac{1+2s}{4}}B_j)
\end{align*}}
and
{\footnotesize\begin{align*}\left(\frac{t-s}{t-1/2}\right)\left(\sum_{j=1}^n(A_j\sharp_{s}B_j)\circ \sum_{j=1}^n(A_j\sharp_{1-s}B_j)
-\sum_{j=1}^n(A_j\sharp B_j)\circ \sum_{j=1}^n(A_j\sharp B_j)\right)
\end{align*}}
are positive operators. So inequality \eqref{maman2} is another refinement of the second inequality of inequalities
\eqref{34rf} and \eqref{moj-mo}.
\end{remark}
If we put $B_j=I\,\,(1\leq j\leq n)$ in Theorem \ref{okmn345}, then we get the next result.
\begin{corollary}
Let $A_j\in{\mathbb B}({\mathscr H})_+\,\,(1\leq j\leq n)$ and either $1\geq t\geq s>{\frac{1}{2}}$ or $0\leq t\leq s<\frac{1}{2}$. Then
\begin{align*}
&\sum_{j=1}^nA_j^{1-s}\circ \sum_{j=1}^nA_j^s
\\&\,\,+\left(\frac{t-s}{t-1/2}\right)\left(\sum_{j=1}^nA_j^{1-s}\circ \sum_{j=1}^nA_j^s
-\sum_{j=1}^nA_j^{\frac{1}{2}}\circ \sum_{j=1}^nA_j^{\frac{1}{2}}\right)\nonumber
\\&\,\,+r'\Big(\sum_{j=1}^nA_j^{1-s}\circ \sum_{j=1}^nA_j^s
+\sum_{j=1}^nA_j^{\frac{1}{2}}\circ \sum_{j=1}^nA_j^{\frac{1}{2}}\\&\,\,\,\,-2\sum_{j=1}^nA_j^{\frac{1+2s}{4}}\circ \sum_{j=1}^nA_j^{\frac{3-2s}{4}}\Big)
\\&\leq \sum_{j=1}^nA_j^{1-t}\circ \sum_{j=1}^nA_j^t\,,
 \end{align*}
where $r'=\min\left\{\frac{t-s}{t-1/2},\frac{s-1/2}{t-1/2}\right\}$.
\end{corollary}
{If in Theorem \ref{okmn345} we replace $A_j, B_j, s, t$, by $x_j, y_j, {s+1\over2}, {t+1\over2}\,\,(1\leq j\leq n)$, respectively, then we reach another refinement of inequality \eqref{choshi}.
\begin{corollary}
Let $x_j,y_j\,\,(1\leq j\leq n)$ be positive numbers and  either  $-1\leq t\leq s<0$ or $1\geq t\geq s>0$. Then
\begin{align*}
&\sum_{j=1}^nx_j^{1+s\over2}y_j^{1-s\over2} \sum_{j=1}^nx_j^{1-s\over2}y_j^{1+s\over2}
\\&\,\,+\left(\frac{t-s}{t}\right)\left(\sum_{j=1}^nx_j^{1+s\over2}y_j^{1-s\over2} \sum_{j=1}^nx_j^{1-s\over2}y_j^{1+s\over2}
-\left(\sum_{j=1}^nx_j^{1\over2}y_j^{1\over2}\right)^2\right)
\\&\,\,+r'\Big(\sum_{j=1}^nx_j^{1+s\over2}y_j^{1-s\over2} \sum_{j=1}^nx_j^{1-s\over2}y_j^{1+s\over2}
+\left(\sum_{j=1}^nx_j^{1\over2}y_j^{1\over2}\right)^2\\&\,\,\,\,-2\sum_{j=1}^nx_j^{2+s\over4}y_j^{2-s\over4} \sum_{j=1}^nx_j^{2-s\over4}y_j^{2+s\over4}\Big)
\\&\leq \sum_{j=1}^nx_j^{1+t\over2}y_j^{1-t\over2} \sum_{j=1}^nx_j^{1-t\over2}y_j^{1+t\over2}\,,
 \end{align*}
where $r'=\min\left\{\frac{t-s}{t},\frac{s-1}{t-1}\right\}$.
\end{corollary}}
\section{Some reverses  of the Callebaut type inequality }

In \cite{W-Z}, the authors showed  a reverse of the  young inequality as follows:
\begin{align}\label{laenana}
 \nu a+(1-\nu)b\leq K\left(\sqrt{\frac{a}{b}},2\right)^{-r'}a^\nu b^{1-\nu}+s\left(
\sqrt{a}-\sqrt{b}\right)^2,
\end{align}
in which $a,b>0$, $\nu\in[0,1]-\left\{\frac{1}{2}\right\}$, $r=\min\left\{\nu,1-\nu\right\}$, $r'=\min\left\{2r,1-2r\right\}$ and $s=\max\left\{\nu,1-\nu\right\}$. Applying \eqref{laenana}  we have the next result.
\begin{lemma}
Let $a,b>0$ and $\nu\in[0,1]-\left\{\frac{1}{2}\right\}$. Then
\begin{align*}
 a+b\leq K\left(\sqrt{\frac{a}{b}},2\right)^{-r'}\left(a^\nu b^{1-\nu}+a^{1-\nu}b^{\nu}\right)+2s\left(
\sqrt{a}-\sqrt{b}\right)^2,
\end{align*}
where   $r=\min\left\{\nu,1-\nu\right\}$, $r'=\min\left\{2r,1-2r\right\}$ and $s=\max\left\{\nu,1-\nu\right\}$. \\In particular, If $\nu\in[0,\frac{1}{2})$, then
\begin{align}\label{now12}
 a+a^{-1}\leq K(a,2)^{-r'}\left(a^{1-2\nu}+a^{-(1-2\nu)}\right)+2(1-\nu)\left(
a^\frac{1}{2}-a^\frac{-1}{2}\right)^2.
\end{align}
\end{lemma}
Now, utilizing  inequality \eqref{now12} and the same argument  in the proof of Lemma \ref{rar32} we can accomplish the corresponding
result.
\begin{lemma}\label{dear}
Let $0<m' \leq B\leq m <M \leq A\leq M'$ and either $1\geq t\geq s>{\frac{1}{2}}$ or $0\leq t\leq s<\frac{1}{2}$.  Then
\begin{align*}
A^{t}\otimes B^{1-t}+A^{1-t}\otimes B^{t}&\leq K\left(\frac{M^{2t-1}}{m^{2t-1}},2\right)^{-r'}
(A^{s}\otimes B^{1-s}+A^{1-s}\otimes B^{s}) \nonumber\\&+\left(\frac{s-1/2}{t-1/2}\right)\left(A^{t}\otimes B^{1-t}+A^{1-t}\otimes B^{t}-2(A^{\frac{1}{2}}\otimes B^{\frac{1}{2}})\right)\,,
 \end{align*}
 where $r'=\min\left\{\frac{t-s}{t-1/2},\frac{s-1/2}{t-1/2}\right\}$.
\end{lemma}
As a consequence of Lemma \ref{dear} we have the following result.
\begin{theorem}\label{mainth}
Let $0<m' \leq B_j\leq m <M \leq A_j\leq M'\,\,(1\leq j\leq n)$.  Then
\begin{align*}
\sum_{j=1}^n(A_j&\sharp_{t}B_j)\circ \sum_{j=1}^n(A_j\sharp_{1-t} B_j)\leq K\left(\frac{M^{2t-1}}{m^{2t-1}},2\right)^{-r'}
\sum_{j=1}^n(A_j\sharp_{s}B_j)\circ \sum_{j=1}^n(A_j\sharp_{1-s}B_j)
\\&\,\,+\left(\frac{s-1/2}{t-1/2}\right)\left(\sum_{j=1}^n(A_j\sharp_{t}B_j)\circ \sum_{j=1}^n(A_j\sharp_{1-t}B_j)
-\sum_{j=1}^n(A_j\sharp B_j)\circ \sum_{j=1}^n(A_j\sharp B_j)\right)
\end{align*}
for either $1\geq t\geq s>{\frac{1}{2}}$ or $0\leq t\leq s<\frac{1}{2}$ and $r'=\min\left\{\frac{t-s}{t-1/2},\frac{s-1/2}{t-1/2}\right\}$.
\end{theorem}
\begin{remark}
If we put $t=1$ and $1\geq s>{\frac{1}{2}}$ in Theorem \ref{mainth}, then we get a reverse of the third inequality of  \eqref{34rf}
\begin{align*}
\left(\sum_{j=1}^nA_j \right)\circ& \left(\sum_{j=1}^n B_j\right)\leq K\left(\frac{M}{m},2\right)^{-r'}
\sum_{j=1}^n(A_j\sharp_{s}B_j)\circ \sum_{j=1}^n(A_j\sharp_{1-s}B_j)
\\&\,\,+\left({2s-1}\right)\left(\left(\sum_{j=1}^nA_j \right)\circ \left(\sum_{j=1}^n B_j\right)
-\sum_{j=1}^n(A_j\sharp B_j)\circ \sum_{j=1}^n(A_j\sharp B_j)\right),
\end{align*}
where $r'=\min\left\{{2-2s},{2s-1}\right\}$.
\end{remark}

{If  we replace $A_j, B_j, s, t$, by $x_j, y_j, {s+1\over2}, {t+1\over2}\,\,(1\leq j\leq n)$ in Theorem \ref{mainth}, respectively, then we reach a reverse of the second inequality of  \eqref{choshi}
\begin{corollary}
Let $0<m' \leq y_j\leq m <M \leq x_j\leq M'\,\,(1\leq j\leq n)$.  Then
\begin{align*}
\sum_{j=1}^nx_j^{1+t\over2}y_j^{1-t\over2}& \sum_{j=1}^nx_j^{1-t\over2}y_j^{1+t\over2}\leq K\left(\frac{M^{2t-1}}{m^{2t-1}},2\right)^{-r'}
\sum_{j=1}^nx_j^{1+s\over2}y_j^{1-s\over2} \sum_{j=1}^nx_j^{1-s\over2}y_j^{1+s\over2}
\\&\,\,+\left(\frac{s-1/2}{t-1/2}\right)\left(\sum_{j=1}^nx_j^{1+t\over2}y_j^{1-t\over2} \sum_{j=1}^nx_j^{1-t\over2}y_j^{1+t\over2}
-\left(\sum_{j=1}^nx_j^{1\over2}y_j^{1\over2}\right)^2\right)
\end{align*}
for either $1\geq t\geq s>{\frac{1}{2}}$ or $0\leq t\leq s<\frac{1}{2}$ and $r'=\min\left\{\frac{t-s}{t-1/2},\frac{s-1/2}{t-1/2}\right\}$.
\end{corollary}}
\begin{proposition}
Let $0<m' \leq B_j\leq m <M \leq A_j\leq M'\,\,(1\leq j\leq n)$ and either $1\geq t\geq s>{\frac{1}{2}}$ or $0\leq t\leq s<\frac{1}{2}$.  Then
\begin{align*}
K\left({h^{2t-1}},2\right)^{r'}&
\sum_{j=1}^n(A_j\sharp_{s}B_j)\circ \sum_{j=1}^n(A_j\sharp_{1-s}B_j)
+\left(\frac{t-s}{t-1/2}\right)\left(\sqrt{h}-
\sqrt{\frac{1}{h}}\right)^2\\&\leq\sum_{j=1}^n(A_j\sharp_{t}B_j)\circ \sum_{j=1}^n(A_j\sharp_{1-t} B_j)\\&\leq K\left({h^{2t-1}},2\right)^{-r'}
\sum_{j=1}^n(A_j\sharp_{s}B_j)\circ \sum_{j=1}^n(A_j\sharp_{1-s}B_j)
\\&\qquad+\left(\frac{s-1/2}{t-1/2}\right)\left(\sqrt{h'}-\sqrt{\frac{1}{h'}}\right)^2,
\end{align*}
where $h=\frac{M}{m}$, $h'=\frac{M'}{m'}$ and $r'=\min\left\{\frac{t-s}{t-1/2},\frac{s-1/2}{t-1/2}\right\}$.
\end{proposition}
\begin{proof}
Since the function  $f(a)=a-\frac{1}{a}$ is increasing on $(0,\infty)$, we have
\begin{align*}
\left(\sqrt{h}-\sqrt{\frac{1}{h}}\right)^2 \leq \left(\sqrt{ a}-\sqrt{\frac{1}{a}}\right)^2\leq\left(\sqrt{h'}-\sqrt{\frac{1}{h'}}\right)^2\,\,\,(h\leq a\leq h').
\end{align*}
 Applying inequalities \eqref{kokol}, \eqref{now12} and the same argument in the proof of Theorem \ref{vow13} we get the desired result.
 \end{proof}

\textbf{Acknowledgement.} The author would like to sincerely thank the anonymous  referee for some useful comments and suggestions. The author also would like to thank the Tusi Mathematical Research Group (TMRG).
%===================================================================================================================================


\begin{thebibliography}{99}

\bibitem{aldaz} J.M.  Aldaz, S.  Barza, M. Fujii and  M.S. Moslehian, \textit{ Advances in operator Cauchy--Schwarz inequalities and their reverses}. Ann. Funct. Anal. \textbf{6} (2015), no. 3, 275--295.

\bibitem{ABM} Lj. Arambasi\'c, D. Baki\'c and M.S. Moslehian, \textit{A treatment of the Cauchy--Schwarz inequality in $C^*$-modules}, J. Math. Anal. Appl. \textbf{381} (2011) 546--556.

\bibitem{mo-moj} M. Bakherad and M.S. Moslehian, \textit{Complementary and refined inequalities of Callebaut inequality}, Linear Multilinear Algebra \textbf{63} (2015), no. {8}, 1678--1692.

\bibitem{CAL} D.K. Callebaut, \textit{Generalization of the Cauchy--Schwarz inequality}, J. Math. Anal. Appl. \textbf{12} (1965), 491–-494.

\bibitem{I-V} D. Ili\v{s}evi\'c and S. Varo\v{s}anec, \textit{On the Cauchy--Schwarz inequality and its reverse in
semi-inner product $C^*$-modules}, Banach J. Math. Anal. \textbf{1} (2007), 78--84.

\bibitem{caleba}M.S. Moslehian, J.S. Matharu and J.S. Aujla, \textit{Non-commutative Callebaut inequality},
Linear Algebra Appl. \textbf{436} (2012) no. 9, 3347--3353.

\bibitem{paul}V.I. Paulsen, \textit{Completely bounded maps and dilation}, Pitman Research Notes in Mathematics Series, \textbf{146}, John Wiley $\&$ Sons, Inc., New York, 1986.

\bibitem{salemi}A. Salemi and A. Sheikh Hosseini, \textit{On reversing of the modified Young inequality}, Ann. Funct. Anal. \textbf{5} (2014), no. 1, 70--76.

\bibitem{W-Z} J. Wu and J. Zhao, \textit{Operator inequalities and reverse
inequalities related to the Kittaneh--Manasrah inequalities}, Linear Multilinear Algebra \textbf{62} (2014), no. 7, 884--894.

\bibitem{wada} S. Wada, \textit{On some refinement of the Cauchy--Schwarz inequality}, Linear Algebra Appl. \textbf{420} (2007) no. 2-3, 433--440.
\bibitem{zho} H. Zuo, G. Shi and  M. Fujii,  \textit{Refined Young inequality with Kantorovich constant}, J. Math. Inequal.,
2011, \textbf{5}, 551--556.

\bibitem{zha} J. Zhao and J.  Wu, \textit{Operator inequalities involving improved Young and its reverse inequalities}, J. Math. Anal. Appl. \textbf{421} (2015), no. 2, 1779--1789.
\end{thebibliography}
\end{document}